\documentclass[11pt]{article}
\usepackage{amsmath,amsthm,verbatim,amssymb,amsfonts,amscd, graphicx, fancyhdr, enumerate}
\usepackage{parskip}%create a blank line between paragraphs without indentation

\usepackage[utf8]{inputenc} %write strange character
\usepackage[T1]{fontenc} 

\makeatletter
\def\thm@space@setup{%
	\thm@preskip=\parskip \thm@postskip=0pt
}
\makeatother %make parskip useful for theorems, definitions...

\theoremstyle{plain} \numberwithin{equation}{section}
\newtheorem{theorem}{Theorem}[section]
\newtheorem{corollary}[theorem]{Corollary}

\newtheorem*{conjecture*}{Conjecture}
\newtheorem{lemma}[theorem]{Lemma}

\newtheorem*{question*}{Question} \topmargin-2cm
\theoremstyle{definition}
\newtheorem{definition}[theorem]{Definition}

\newtheorem{remark}[theorem]{Remark}

\usepackage[textheight=23cm, textwidth=16.5cm, footskip=1.0cm]{geometry}

\DeclareMathOperator{\Hessian}{Hess}
\DeclareMathOperator{\distance}{dist}

\renewcommand{\Re}{\operatorname{Re}}

\usepackage[
backend=biber,  %(arxiv) after swiching to bibtex, run it in sharelatex. upload the bbl of sharelatex to arxiv
natbib=true,
url=false, 
]{biblatex}
\begin{document}
	
	\title{The Diederich--Forn\ae ss index I: for domains of non-trivial index}

	\author{
		Bingyuan Liu\\ bingyuan@ucr.edu
	}
	
	\date{March 9, 2017}

	\maketitle
	
	\begin{abstract}
We study bounded pseudoconvex domains in complex Euclidean space. We define an index associated to the boundary and show this new index is equivalent to the Diederich--Forn\ae ss index defined in 1977. This connects the Diederich--Forn\ae ss index to boundary conditions and refines the Levi pseudoconvexity. We also prove the $\beta$-worm domain is of index $\pi/{(2\beta)}$. It is the first time that a precise non-trivial Diederich--Forn\ae ss index in Euclidean spaces is obtained. This finding also indicates that the Diederich--Forn\ae ss index is a continuum in $(0,1]$, not a discrete set. The ideas of proof involve a new complex geometric analytic technique on the boundary and detailed estimates on differential equations.
	\end{abstract}
	
	\section{Introduction}\label{sec0}	
	The root of modern complex analysis is the theory of pseudoconvexity. In 1906, Hartogs discovered a surprising phenomenon in $\mathbb{C}^n$ for $n>1$ which can never appear in $\mathbb{C}$. He found that on some bounded domains in $\mathbb{C}^n$ for $n>1$, every holomorphic function can be holomorphically extended to a larger domain. After that, Levi and Hartogs studied this phenomenon and derived two different necessary and sufficient descriptions of the domains where this phenomenon does not hold. These two conditions are equivalent and are known as Levi pseudoconvexity and Hartogs pseudoconvexity. Their equivalence is represented in the following theorem.
	
	\begin{theorem}[\citep{Kr01}]
		Let $\Omega$ be a bounded domain with $C^2$ boundary in $\mathbb{C}^n$. Let \[\delta(z):=\begin{cases}
		-\distance(z, \partial\Omega) & z\in\Omega\\
		\distance(z, \partial\Omega) & \text{otherwise}.
		\end{cases}\]  be a \textit{signed distance function} of $\Omega$. Then the two following statements are equivalent:
		\begin{enumerate}
			\item The domain $\Omega$ is Hartogs pseudoconvex, i.e., $-\log(-\delta)$ is plurisubharmonic in $\Omega$.
			\item The domain $\Omega$ is Levi pseudoconvex, i.e., $\Hessian_\delta (L, L)\geq 0$ for all $(1,0)$ tangent vector field $L$ of $\partial\Omega$.
		\end{enumerate}
	\end{theorem}

The theorem essentially states how a boundary condition affects the nature of domains. It was soon discovered that the pseudoconvexity is definitive to many other fundamental problems. In complex Hodge theory, Garabedian--Spencer \citep{GS52} suggested the $\bar{\partial}$-Neumann problem on bounded pseudoconvex domains with smooth boundary. The $L^2$ existence theory and the global regularity of $\bar{\partial}$-Neumann problem on strongly pseudoconvex domains (i.e. $\Hessian_\delta (L, L)$ strictly greater than $0$ everywhere) were solved by Kohn \citep{Ko63}, \citep{Ko64} and H\"{o}rmander \citep{Ho65} during 1940s-1960s. However, the global regularity for many other pseudoconvex domains (with smooth boundary) is still lacking a complete answer (see Boas--Straube \citep{BS99}). In particular, the $\bar{\partial}$-Neumann operator may or may not preserve the sobolev space $W^s_{(p, q)}(\overline{\Omega})$ of $(p,q)$-forms for all $s>0$ (see Barrett \citep{Ba92},  Kiselman \citep{Ki91} and Boas--Straube \citep{BS93}) even under the category of smoothly bounded pseudoconvexity. This demands a refinement of pseudoconvexity to further classify the (weakly) pseudoconvex domains.

In 1977, Diederich--Forn\ae ss proved a celebrated theorem in \citep{DF77b}. The idea of the theorem is to replace the classical $\log$ composition in the notion of Hartogs pseudoconvexity with a power function composition. This classifies the notion of pseudoconvexity in the detail.
\begin{definition}\label{df}
	Let $\Omega$ be a bounded pseudoconvex domain with $C^2$ boundary in $\mathbb{C}^n$. The function $\rho$ is called a \textit{defining function} if following conditions are satisfied:
	\begin{enumerate}
		\item the function $\rho$ is $C^2$ on a neighborhood of $\overline{\Omega}$,
		\item the domain $\Omega=\lbrace z\in\mathbb{C}^n: \rho(z)<0 \rbrace$, and 
		\item the gradient $\nabla\rho\neq 0$ on $\partial\Omega$.
	\end{enumerate}
	The number 
	$0 < \tau_\rho < 1$ is called a \textit{Diederich-Forn\ae ss exponent} if there exists a defining 
	function $\rho$ of $\Omega$ so that $-(-\rho)^{\tau_\rho}$ is plurisubharmonic in $\Omega$. 
	The index 
	\[\eta:=\sup \tau_\rho ,\] 
	where the supremum is taken over all defining functions of $\Omega$, is 
	called the \textit{Diederich-Forn\ae ss index} of the domain $\Omega$ (see Chen--Fu \citep{CF11}). 
\end{definition}

The index does a fundamental job to refine the notion of Hartogs pseudoconvexity in terms of the Sobolev regularity for the $\bar{\partial}$-Neumann operator and Bergman projection. For example, Berndtsson--Charpentier showed in \citep{BC00}, that the $\bar{\partial}$-Neumann operator and the Bergman projection preserves $W^k(\overline{\Omega})$ for $k<\eta_0/2$ when the Diederich--Forn\ae ss index is $\eta_0$. However, the Diederich--Forn\ae ss index has not been understood thoroughly because the verifications and computations are rather difficult. In this paper, we understand the index by refining the notion of Levi pseudoconvexity. In other words, we define a natural index of the boundary and show this new index is equal to the Diederich--Forn\ae ss index. This is given by Theorem \ref{equiv}. This theorem can be thought of as a refinement of the equivalence of Hartogs pseudoconvexity and Levi pseudoconvexity. 

To prove the theorem, we first simplify some common quantities and establish several useful identities. Many of the simplifications are motivated by the geometric analysis for Riemannian and K\"{a}hler manifolds. We also pass the plurisubharmonicity to the study of quadratic equations. Using the useful identities and the discriminant for quadratic equations, we are able to obtain the theorem of equivalence. The ideas of proof involve new complex analytic techniques on the boundary.

By the completeness of Diederich--Forn\ae ss index in terms of Levi's notion, we understand the Diederich--Forn\ae ss index is fundamental to the domains in complex analysis. Consequently, we will say the $\Omega$ is a \textit{pseudoconvex domain of index $\eta_0$} when the Diederich--Forn\ae ss index of $\Omega$ is $\eta_0$. If $\eta_0\in (0,1)$, we say that the $\Omega$ is of \textit{non-trivial index}. Otherwise, we say it is of \textit{trivial index}. The Diederich--Forn\ae ss index gives an alternative classification to the classical notions (e.g. strongly pseudoconvex, finite type, et al.).

The theorem of equivalence makes the computation of the index simpler. As one application, we will compute the index of $\beta$-worm domains (see Definition \ref{worm}). The $\beta$-worm domains $\Omega_\beta$ have been introduced by Diederich--Forn\ae ss in \citep{DF77a}. These domains are the most famous pseudoconvex domains with smooth boundary on which the $\bar{\partial}$-Neumann operators and Bergman projections do not preserve $W^k(\overline{\Omega}_\beta)$ for all $k>0$. Indeed,  the non-preservation on $W^k(\overline{\Omega}_\beta)$ for $k\geq\pi/{(2\beta-\pi)}$ was showed by Barrett in 1992 (see \citep{Ba92}). On the other hand, the domains $\Omega_\beta$ are the only known domains of non-trivial index in $\mathbb{C}^n$. Since the 1990s, it has been known that the index of $\Omega_\beta$ is less or equal to $2\pi/(2\beta-\pi)$. In 2008, Krantz--Peloso \citep{KP08} improved this upper bound to $\pi/(2\beta-\pi)$. As an application of Theorem \ref{equiv}, we improve the upper bound of Diederich--Forn\ae ss index of $\Omega_\beta$ to $\pi/(2\beta)$. We also show this bound is sharp. In other words, we determine that $\Omega_\beta$ is of index $\pi/(2\beta)$. This is the first time that an accurate non-trivial Diederich--Forn\ae ss index in Euclidean spaces is obtained. See Theorem \ref{calcworm}. 

To find the Diederich--Forn\ae ss index of $\beta$-worm domains, we pass the Diederich--Forn\ae ss index to the solvability of a partial differential equation defined on an annulus. This requires the use of our new index. We then associate the solvability of the aforementioned partial differential equation with the solvability of the famous Riccati equations. Consequently, the Diederich--Forn\ae ss indexes of $\beta$-worm domains are found using the comparison principal of ordinary differential equations. 

We remark that in the case of complex manifolds, Diederich--Ohsawa in \citep{DO85} and \citep{DO07} have constructed a domain with Levi-flat boundary. This domain is known to have the Diederich--Forn\ae ss index $0.5$ by different arguments involving Adachi--Brinkschulte \citep{AB14}, Fu--Shaw \citep{FS14} (see \citep{FS14} for other estimates) and Example 4.5 of Adachi \citep{Ad15}. This type of domain, however does not exist in Euclidean spaces because of the nonexistence of domains with Levi-flat boundary in $\mathbb{C}^n$. 

Since 1977, many mathematicians have been working on this topic. Here are a few important works which we have not mentioned: Demailly \citep{De87}, Forn\ae ss--Herbig \citep{FH08}, Herbig--McNeal \citep{HM12b}, Harrington \citep{Ha08}, and Range \citep{Ra81}.  

 The outline of the paper is as the following. After some preparation in preliminaries, we define the Diederich--Forn\ae ss index of hypersurface in Section \ref{sec2} (see Definition \ref{dfb}). We also show this new index is equivalent to the Diederich--Forn\ae ss index defined above. This equivalence is proved in Theorem \ref{equiv}. In Section \ref{sec3}, we study the $\beta$-worm domains $\Omega_\beta$. Using Theorem \ref{equiv}, we find the exact Diederich--Forn\ae ss index of $\Omega_\beta$ in Theorem \ref{calcworm}. This also shows that the Diederich--Forn\ae ss index is continuum, not discrete.

In this paper, all smoothness can be replaced with $C^3$.

	\section{Preliminaries}\label{sec1}
	
	In this paper, we use the following terminology. Let $\Omega$ be a bounded pseudoconvex domain with smooth boundary in $\mathbb{C}^n$. Recall $\delta$ is the signed distance function of $\Omega$ defined in Section \ref{sec0}. We  denote a fixed defining function $\Omega$ by $r$. We let $\rho$ be an arbitrary defining function of $\Omega$. It is well-known that $\rho=re^\psi$ for some smooth $\psi$ defined in a neighborhood of $\Omega$. We fix this terminology and will not mention the relation of $\delta$, $r$ and $\psi$ in the rest of article. We denote a tubular neighborhood of $\partial\Omega$ by $U$. We also use terminology from K\"{a}hler geometry. The notation $g$ stands for the standard Euclidean metric of $\mathbb{C}^n$ and $\Hessian_r(L, N)=g(\nabla_L\nabla r, N)$ stands for Hessian with this metric.
	
Let $\Omega$ be a pseudoconvex domain with smooth boundary defined by $r$. We define 
\[N_r=\frac{1}{\sqrt{\sum_{j=1}^{n}|\frac{\partial r}{\partial z_j}|^2}}\sum_{j=1}^{n}\frac{\partial r}{\partial \bar{z}_j}\frac{\partial}{\partial z_j}.\] If we replace $r$ by $\delta$ in the definition above, we can also define the tangent field \[N_\delta=\frac{1}{\sqrt{\sum_{j=1}^{n}|\frac{\partial \delta}{\partial z_j}|^2}}\sum_{j=1}^{n}\frac{\partial \delta}{\partial \bar{z}_j}\frac{\partial}{\partial z_j}\]. 

The reader can check that $N_r$ and $N_\delta$ have the following properties:
\begin{enumerate}
	\item $N_r=N_\delta$ on $\partial\Omega$.
	\item $\sqrt{2}N_\delta$ is unit vector in $\mathbb{C}^n$.
	\item $N_\delta+\overline{N}_\delta=\nabla\delta$ in $\mathbb{C}^n$.
	\item $N_r+\overline{N}_r=\nabla\delta$ on $\partial\Omega$.
	\item $N_\delta\delta =\frac{1}{2}$ and $N_r r=\frac{\|\nabla r\|}{2}$.
\end{enumerate}
		 
		 Here, the notation $\nabla$ stands for the gradient in terms of the Euclidean metric. In this article, whenever we mention $\lim\limits_{z\to p}$, it means $z$ approaches $p$ from the normal direction. We will frequently use the following two basic results.

		 \begin{lemma}\label{basic}
		 	Let $\Omega$ be a bounded pseudoconvex domain with smooth boundary in $\mathbb{C}^n$. Let $L$ be a smooth $(1,0)$-tangent vector field on $\partial\Omega$. 
		 	\begin{enumerate}
		 		\item Then at $p\in \partial\Omega$ \[	\lim\limits_{z\to p}\frac{\overline{L}\rho}{-\rho}=\frac{N_r(\overline{L}\rho)}{-N_r\rho}.\]
		 		\item Assume $\Hessian_\rho(L, L)=0$ at $p\in \partial\Omega$. Then at $p$,	 		\[\begin{split}
		 		&
		 		\lim\limits_{z\to p}\frac{\Hessian_\rho(L, L)}{-\rho}=\frac{N_r\Hessian_\rho(L, L)}{-N_r\rho}.
		 		\end{split}
		 		\] 
		 	\end{enumerate} 
		 	
		 \end{lemma}
	 
\begin{proof}
	On $\partial\Omega$, $N_r-\overline{N}_r$ is tangent to $\partial\Omega$. Consequently we have that $(N-\overline{N}_r)\overline{L}\rho=0$. This implies $N_r(\overline{L}\rho)=\overline{N}_r(\overline{L}\rho)$. 
	
	We are going to show the first conclusion. For the same reason, if $\Hessian_\rho(L, L)=0$ at $p\in \partial\Omega$, we have that $N_r(\Hessian_\rho(L, L))=\overline{N}_r(\Hessian_\rho(L, L))$.
	
	Since $\overline{L}\rho=0$ on $\partial\Omega$, then \[	\lim\limits_{z\to p}\frac{\overline{L}\rho}{-\rho}=\lim\limits_{z\to p}\frac{\overline{L}\rho|_z-\overline{L}\rho|_p}{-\rho|_z+\rho|_p}=\frac{(N_r+\overline{N}_r)(\overline{L}\rho)}{-(N_r+\overline{N}_r)\rho}\Bigg|_p=\frac{N_r(\overline{L}\rho)}{-N_r\rho}\Bigg|_p.\] 
	
	Similarly, we can also obtain the second conclusion.
\end{proof}
 
 \begin{lemma}\label{basic2}
 	Let $\Omega$ be a bounded pseudoconvex domain with smooth boundary in $\mathbb{C}^n$. Suppose $L$ is a $(1,0)$ tangent vector field so that $\Hessian_r(L, L)=0$ at $p\in \partial\Omega$. Assume, $T_j$ for $1\leq j\leq n-2$ are $(1,0)$ tangent vector fields and $L, T_1, T_2, \dots, T_{n-2}$ are orthogonal at $p$. Then $\Hessian_\delta (L, T_j)=0$ for $1\leq j\leq n-2$ at $p$.
 \end{lemma}

\begin{proof}
	If $L=0$, the lemma obviously holds. We assume $L\neq 0$. Without loss of generality, we assume $\lbrace L, T_1, T_2, \dots, T_{n-2}\rbrace$ is an orthonormal basis at $p$. The pseudoconvexity at $p$ implies the matrix \[\begin{pmatrix}
	\Hessian_r(L, L)& \Hessian_r(L, T_1)&\cdots&\Hessian_r (L, T_{n-2})\\
	\Hessian_r(T_1, L)& \Hessian_r(T_1, T_1)&\cdots&\Hessian_r (T_1, T_{n-2})\\
	\vdots&\vdots&\ddots&\vdots\\
\Hessian_r(T_{n-2}, L)& \Hessian_r(T_{n-2}, T_1)&\cdots&\Hessian_r (T_{n-2}, T_{n-2})\\
	\end{pmatrix}\] is semi-positive definite. This is equivalent to say its principal minors are nonnegative. Particularly, for $1\leq j\leq n-2$, we take out the principal minor
	\[\begin{pmatrix}
	\Hessian_r(L, L)&\Hessian_r(L, T_j)\\
		\Hessian_r(T_j, L)&\Hessian_r(T_j, T_j)\\
	\end{pmatrix}.\] The determinant should be nonnegative. That is at $p$, \[\Hessian_r(L, L)\Hessian_r(T_j, T_j)-|\Hessian_r(L, T_j)|^2\geq 0.\] But since $\Hessian_r(L, L)=0$ at $p$, we find out $\Hessian_r(L, T_j)=0$ at $p$.
\end{proof}

	\section{The hypersurface version of Diederich-Forn\ae ss index}\label{sec2}
	
	Let $r$ be an arbitrary defining function of $\Omega$. We know that $\rho=re^\psi$ for some smooth function $\psi$ defined in a neighborhood of $\overline{\Omega}$ in $\mathbb{C}^n$. 
	
	\begin{lemma}\label{calc}
		Let $\Omega$ be a bounded pseudoconvex domain with smooth boundary in $\mathbb{C}^n$. Let $L$ be a smooth $(1,0)$-tangent vector field in a tubular neighborhood of $\partial\Omega$. Let $N$ denote $N_r$.
		\begin{enumerate}
			\item  Then at $p\in \partial\Omega$,	\[
			-2\Re\left(\Hessian_{\rho} (L, N)\cdot\lim\limits_{z\to p}\frac{\overline{L}\rho}{-\rho}\cdot N\rho\right)=
		2e^{2\psi}\left(|L\psi|^2|\overline{N}r|^2+\Re\left((L\psi)(\overline{N}r)\Hessian_r(N, L)\right)\right).
			\]
		\item Assume $\Hessian_\rho(L, L)=0$ at $p\in \partial\Omega$. Then at $p$,	\[\begin{split}
		&\lim\limits_{z\to p}\frac{\Hessian_{\rho} (L, L)}{-\rho}\left|N\rho\right|^2\\=&-e^{2\psi}(Nr)\Bigg(|L\psi|^2(Nr)+(Nr)\Hessian_\psi(L, L)+ g(\nabla_L\nabla_N\nabla r, L)-2\|\nabla r\|^{-1}|\Hessian_r(N, L)|^2\\&+\|\nabla r\|L\left(\frac{1}{\|\nabla r\|}\right)\Hessian_r(N, L)\Bigg).
		\end{split}\]
	\end{enumerate}
	\end{lemma}

\begin{remark}
	The reader should be warned that $L$ might be vanishing at some points.
\end{remark}
	
	\begin{proof}
		First, without loss of generality, we assume that $L$ is defined in a tubular neighborhood of $\partial\Omega$ so that $Lr=0$.
We use the identity from Lemma \ref{basic}. At $p\in \partial\Omega$,
			\[\begin{split}
		&-2\Re\left(\Hessian_{\rho} (L, N)\cdot\frac{\overline{L}\rho}{-\rho}\cdot N\rho\right)\\=
		&2\Re\left(\Hessian_\rho(L, N)\cdot N(\overline{L}\rho)\right)\\=&2\Re\left(\Hessian_\rho(L, N)\cdot (\Hessian_\rho(N, L)+(\nabla_N\overline{L})\rho)\right)\\=&2|\Hessian_\rho(L, N)|^2+2\Re\left(\Hessian_\rho(L, N)\cdot (\nabla_N\overline{L})\rho\right)\\=&2|\Hessian_\rho(L, N)|^2-2e^{2\psi}\Re\left((L\psi)(\overline{N}r)\Hessian_r(N, L)\right)-2e^{2\psi}|\Hessian_r(L, N)|^2\\=&2e^{2\psi}\left(|(L\psi)(\overline{N}r)+\Hessian_r(L, N)|^2-\Re\left((L\psi)(\overline{N}r)\Hessian_r(N, L)\right)-|\Hessian_r(L, N)|^2\right)\\=&2e^{2\psi}\left(|L\psi|^2|\overline{N}r|^2+\Re\left((L\psi)(\overline{N}r)\Hessian_r(N, L)\right)\right).
		\end{split}\]
		
		Now with the assumption $\Hessian_\rho(L, L)=0$ at $p\in\partial\Omega$, we compute \[
		\frac{\Hessian_{\rho} (L, L)}{-\rho}\left|N\rho\right|^2.\]Again, we use the identities from Lemma \ref{basic} to obtain that
		\[
		\lim\limits_{z\to p}\frac{\Hessian_{\rho} (L, L)}{-\rho}\left|N\rho\right|^2=\frac{N\Hessian_{\rho} (L, L)}{-N\rho}\left|N\rho\right|^2.\]
		
		Here \[N\rho=N(r e^\psi)=r e^\psi N\psi+e^\psi Nr=e^\psi(Nr)\] at $\partial\Omega$, hence \[\lim\limits_{z\to p}\frac{\Hessian_{\rho} (L, L)}{-\rho}\left|N\rho\right|^2=-e^\psi (Nr)N\Hessian_{\rho} (L, L).\]
		
		We are going to compute \[N\Hessian_\rho(L, L).\] Since we only consider the Euclidean space $\mathbb{C}^n$ where the curvature tensor is 0, we can obtain that	\[\begin{split}
		&N\Hessian_\rho(L, L)\\=&g(\nabla_N\nabla_L\nabla\rho, L)+g(\nabla_L\nabla\rho, \nabla_{\overline{N}} L)\\=&g(\nabla_L\nabla_N\nabla\rho, L)+g(\nabla_{[N, L]}\nabla\rho, L)+g(\nabla_L\nabla\rho, \nabla_{\overline{N}} L)\\=&Lg(\nabla_N\nabla\rho, L)-g(\nabla_N\nabla\rho, \nabla_{\overline{L}}L)+g(\nabla_{[N, L]}\nabla\rho, L)+g(\nabla_L\nabla\rho, \nabla_{\overline{N}} L)
		\end{split}\]
		
		For the following paragraphs, we are going to compute the four terms generated from the last line one by one.
		
		First, we compute the term $Lg(\nabla_N\nabla\rho, L)$ at $p$.
		\[\begin{split}
	&Lg(\nabla_N\nabla\rho, L)\\=&L\left(g(\nabla_N(r e^\psi\nabla\psi), L)+g(\nabla_N(e^\psi\nabla r), L)\right)\\=&L\left(e^\psi(N r)(\overline{L}\psi)+r e^\psi(N\psi)(\overline{L}\psi)+r e^\psi g(\nabla_N\nabla\psi, L)+e^\psi (N\psi)(\overline{L}r)+e^\psi g(\nabla_N\nabla r, L)\right)\\=&e^\psi (N r)|L\psi|^2+e^\psi(LN r)(\overline{L}\psi)+e^\psi(N r)(L\overline{L}\psi)+e^\psi(L\psi)g(\nabla_N\nabla r, L)+e^\psi Lg(\nabla_N\nabla r, L).
	\end{split}\]
	
	Second, we compute $-g(\nabla_N\nabla\rho, \nabla_{\overline{L}}L)$.
	\[\begin{split}
	&-g(\nabla_N\nabla\rho, \nabla_{\overline{L}}L)\\=&-g(\nabla_N(e^\psi\nabla r), \nabla_{\overline{L}}L)-g(\nabla_N(r e^\psi\nabla\psi),\nabla_{\overline{L}} L)\\=&-e^\psi g(\nabla_N\nabla r, \nabla_{\overline{L}}L)-e^\psi(Nr)(\nabla_L {\overline{L}})\psi.
	\end{split}\] Here we used the fact that $(\nabla_{\overline{L}}L)r=0$. This is because that \[0=\Hessian_r(L, L)=\overline{L}(Lr)-(\nabla_{\overline{L}}L)r=-(\nabla_{\overline{L}}L)r.\]
		
		Third, we compute $g(\nabla_L\nabla\rho, \nabla_{\overline{N}} L)$.
		\[\begin{split}
		&g(\nabla_L\nabla\rho, \nabla_{\overline{N}} L)\\=&g(\nabla_L(e^\psi\nabla r), \nabla_{\overline{N}} L)+g(\nabla_L(r e^\psi\nabla\psi), \nabla_{\overline{N}} L)\\=&e^\psi(L\psi)(\nabla_N\overline{L})r+e^\psi g(\nabla_L\nabla r, \nabla_{\overline{N}}L)\\=&-e^\psi(L\psi)\Hessian_r(N, L)+e^\psi g(\nabla_L\nabla r, \sqrt{2}N)g(\sqrt{2}N, \nabla_{\overline{N}}L)\\=&-e^\psi(L\psi)\Hessian_r(N, L)+2e^\psi g(\nabla_L\nabla r, N)g(N+\overline{N}, \nabla_{\overline{N}}L)\\=&-e^\psi(L\psi)\Hessian_r(N, L)+2\|\nabla r\|^{-1}e^\psi g(\nabla_L\nabla r, N)g(\nabla r, \nabla_{\overline{N}}L)\\=&-e^\psi(L\psi)\Hessian_r(N, L)-2\|\nabla r\|^{-1}e^\psi g(\nabla_L\nabla r, N)g(\nabla_N\nabla r, L)\\=&-e^\psi(L\psi)\Hessian_r(N, L)-2e^\psi \|\nabla r\|^{-1}|\Hessian_r(N, L)|^2.
		\end{split}\] Here we used the fact from Lemma \ref{basic2} that $\Hessian_r(L, T_j)=0$ once $L, T_1, \cdots, T_{n-2}$ are orthogonal.
		
		Finally, for the same reason, we compute $g(\nabla_{[N, L]}\nabla\rho, L)$.	
		\[\begin{split}
		&g(\nabla_{[N, L]}\nabla\rho, L)\\=&g(\nabla_{[N, L]}(r e^\psi\nabla\psi), L)+g(\nabla_{[N, L]}(e^\psi\nabla r), L)\\=&e^\psi([N, L]r)(\overline{L}\psi)+e^\psi g(\nabla_{[N, L]}\nabla r, L)\\=&-e^\psi(LNr)(\overline{L}\psi)+\sqrt{2}e^\psi g([N, L], \sqrt{2}N)g(\nabla_N\nabla r, L)\\=&-e^\psi(LNr)(\overline{L}\psi)+e^\psi\|\nabla r\|L\left(\frac{1}{\|\nabla r\|}\right)\Hessian_r(N, L).
		\end{split}\]
		
		The last equation is because  \[g([N, L], N)=g([N, L], N+\overline{N})=g([N, L], \nabla\delta)= [N, L]\delta=NL\delta=\frac{\|\nabla r\|}{2}L\left(\frac{1}{\|\nabla r\|}\right).\] 
		
		Hence, 	\[\begin{split}
		&\lim\limits_{z\to p}\frac{\Hessian_{\rho} (L, L)}{-\rho}\left|N\rho\right|^2\\=&-e^{2\psi}(Nr)\Bigg(|L\psi|^2(Nr)+(Nr)\Hessian_\psi(L, L)+ g(\nabla_L\nabla_N\nabla r, L)-2\|\nabla r\|^{-1}|\Hessian_r(N, L)|^2\\&+\|\nabla r\|L\left(\frac{1}{\|\nabla r\|}\right)\Hessian_r(N, L)\Bigg).
		\end{split}\]
	\end{proof}
	
	By the preceding lemma, assume that $\Hessian_\rho(L, L)=0$ at $p\in \partial\Omega$. At $p$, we obtain	
	\[\begin{split}
		&-2\Re\left(\Hessian_{\rho} (L, N)\cdot\lim\limits_{z\to p}\frac{\overline{L}\rho}{-\rho}\cdot N\rho\right)+\lim\limits_{z\to p}\frac{\Hessian_{\rho} (L, L)}{-\rho}\left|N\rho\right|^2\\=&e^{2\psi}\Bigg(|(\overline{L}\psi)(Nr)+\Hessian_r(N, L)|^2-\frac{\|\nabla r\|}{2}\Big(\frac{\|\nabla r\|}{2}\Hessian_\psi(L, L)+g(\nabla_L\nabla_N\nabla r, L)\\&+\|\nabla r\|L\left(\frac{1}{\|\nabla r\|}\right)\Hessian_r(N, L)\Big)\Bigg).
	\end{split}
	\]
	
	We can see that the last term $\|\nabla r\|L\left(\frac{1}{\|\nabla r\|}\right)\Hessian_r(N, L)$ can be dropped if the $\|\nabla r\|$ is constant. So in practice, we choose the $r$ with constant gradient norm in order to make the calculation simple.

	From now on, let $U$ be a tubular neighborhood of $\partial\Omega$. 
	\begin{lemma}\label{2.2}
			Let $\Omega$ be a bounded pseudoconvex domain with smooth boundary in $\mathbb{C}^n$. Let $L$ be a smooth $(1,0)$-tangent vector field in $U\cap\overline{\Omega}$ so that $Lr=0$. We assume $\Hessian_\rho(L, L)=0$ at $p\in \partial\Omega$. We also denote $N_r$ by $N$.  Consider the following expression:
			\begin{equation}\label{star}
		\begin{split}
					&\Hessian_{-(-\rho)^\eta }(aL+bN,aL+bN)\\=&|a|^2\Hessian_{-(-\rho)^\eta}(L, L)+|b|^2\Hessian_{-(-\rho)^\eta }(N, N)+2\Re (a\bar{b}\Hessian_{-(-\rho)^\eta }(L, N))\\=&\eta (-\rho)^{\eta-1}\Bigg(|a|^2\Big(\Hessian_\rho (L, L)+\frac{1-\eta}{-\rho}|L\rho|^2\Big)+2\Re \Big(a\bar{b}\Big(\Hessian_\rho (L, N)+\frac{1-\eta}{-\rho}L\rho\cdot\overline{N}\rho\Big)\Big)\\&+|b|^2 \Big(\Hessian_\rho (N, N)+\frac{1-\eta}{-\rho}N\rho\overline{N} \rho\Big)\Bigg).
				\end{split}
			\end{equation}
		If (\ref{star}) is positive for all $(a,b)\in\mathbb{C}^2\backslash(0,0)$ and all $z\in U\cap\Omega$, then
			\[\begin{split}
			&\left(\frac{1}{1-\eta}-1\right)\left|(\overline{L}\psi)(Nr)+\Hessian_r(N, L)\right|^2+\frac{\|\nabla r\|}{2}\Bigg(\frac{\|\nabla r\|}{2}\Hessian_\psi(L, L)+g(\nabla_L\nabla_N\nabla r, L)\\&+\|\nabla r\|L\left(\frac{1}{\|\nabla r\|}\right)\Hessian_r(N, L)\Bigg)\leq 0.\end{split}
			\] at $p$. 
	\end{lemma}

\begin{proof}
	First, let us prove that \[\begin{split}
	&I:=|a|^2\Big(\Hessian_\rho (L, L)+\frac{1-\eta}{-\rho}|L\rho|^2\Big)+2\Re \Big(a\bar{b}\Big(\Hessian_\rho (L, N)+\frac{1-\eta}{-\rho}L\rho\cdot\overline{N}\rho\Big)\Big)\\&+|b|^2 \Big(\Hessian_\rho (N, N)+\frac{1-\eta}{-\rho}N\rho\overline{N} \rho\Big)
	\end{split}\] is positive for all $(a,b)\in\mathbb{C}^2\backslash(0,0)$ and all $z\in U\cap\Omega$ if and only if \[\begin{split}
	&I\!I:=|a|^2\Big(\Hessian_\rho (L, L)+\frac{1-\eta}{-\rho}|L\rho|^2\Big)-2\Big|a\bar{b}\Big(\Hessian_\rho (L, N)+\frac{1-\eta}{-\rho}L\rho\cdot\overline{N}\rho\Big)\Big|\\&+|b|^2 \Big(\Hessian_\rho (N, N)+\frac{1-\eta}{-\rho}N\rho\overline{N} \rho\Big)
	\end{split}\] is positive for all $(a,b)\in\mathbb{C}^2\backslash(0,0)$ and all $z\in U\cap\Omega$. Since \[I\!I\leq I\]  we only need to show that $I$ being positive for all $(a,b)\in\mathbb{C}^2\backslash(0,0)$ and all $z\in U\cap\Omega$ implies $I\!I$ is positive for all $(a,b)\in\mathbb{C}^2\backslash(0,0)$ and all $z\in U\cap\Omega$. Suppose $I$ is positive for all $(a,b)\in\mathbb{C}^2\backslash(0,0)$. In particular, this includes the case of $a, b$ satisfying the following identity at each point $z\in U\cap\Omega$, \[\Re \Big(a\bar{b}\Big(\Hessian_\rho (L, N)+\frac{1-\eta}{-\rho}L\rho\cdot\overline{N}\rho\Big)\Big)=-\Big|a\bar{b}\Big(\Hessian_\rho (L, N)+\frac{1-\eta}{-\rho}L\rho\cdot\overline{N}\rho\Big)\Big|.\] Hence $I$ being positive for all $(a,b)\in\mathbb{C}^2\backslash(0,0)$ and all $z\in U\cap\Omega$ implies $I\!I$ is positive for all $(a,b)\in\mathbb{C}^2\backslash(0,0)$ and all $z\in U\cap\Omega$. This gives that 	\[\begin{split}
	&\Hessian_{-(-\rho)^\eta }(aL+bN,aL+bN)\\=&|a|^2\Hessian_{-(-\rho)^\eta}(L, L)+|b|^2\Hessian_{-(-\rho)^\eta }(N, N)+2\Re (a\bar{b}\Hessian_{-(-\rho)^\eta }(L, N))\\=&\eta (-\rho)^{\eta-1}\Bigg(|a|^2\Big(\Hessian_\rho (L, L)+\frac{1-\eta}{-\rho}|L\rho|^2\Big)+2\Re \Big(a\bar{b}\Big(\Hessian_\rho (L, N)+\frac{1-\eta}{-\rho}L\rho\cdot\overline{N}\rho\Big)\Big)\\&+|b|^2 \Big(\Hessian_\rho (N, N)+\frac{1-\eta}{-\rho}N(\rho)\overline{N} (\rho)\Big)\Bigg)
	\end{split}\]
	is positive for all $(a,b)\in\mathbb{C}^2\backslash(0,0)$ and all $z\in U\cap\Omega$ if and only if \[\begin{split}
	&I\!I:=|a|^2\Big(\Hessian_\rho (L, L)+\frac{1-\eta}{-\rho}|L\rho|^2\Big)-2\Big|a\bar{b}\Big(\Hessian_\rho (L, N)+\frac{1-\eta}{-\rho}L\rho\cdot\overline{N}\rho\Big)\Big|\\&+|b|^2 \Big(\Hessian_\rho (N, N)+\frac{1-\eta}{-\rho}N(\rho)\overline{N} (\rho)\Big)
	\end{split}\] is positive for all $(a,b)\in\mathbb{C}^2\backslash(0,0)$ and all $z\in U\cap\Omega$, because $\eta(-\rho)^{\eta-1}$ is always positive. From now on, we focus on \[\begin{split}
	&|a|^2\Big(\Hessian_\rho (L, L)+\frac{1-\eta}{-\rho}|L\rho|^2\Big)-2\Big|a\bar{b}\Big(\Hessian_\rho (L, N)+\frac{1-\eta}{-\rho}L\rho\cdot\overline{N}\rho\Big)\Big|\\&+|b|^2 \Big(\Hessian_\rho (N, N)+\frac{1-\eta}{-\rho}N(\rho)\overline{N} (\rho)\Big)>0.
	\end{split}\] Without loss of generality, we assume that $\Hessian_\rho (N, N)+\frac{1-\eta}{-\rho}N(\rho)\overline{N} (\rho)>0$ in $U$. This is because that this term will blow up when $z$ approaches $\partial\Omega$. If $a=0$, the inequality above holds trivially. When $a\neq 0$, we derive a new inequality by dividing both sides by $|a|^2$: \[\begin{split}
	&\Big(\Hessian_\rho (L, L)+\frac{1-\eta}{-\rho}|L\rho|^2\Big)-2\left|\frac{b}{a}\right|\Big|\Big(\Hessian_\rho (L, N)+\frac{1-\eta}{-\rho}L\rho\cdot\overline{N}\rho\Big)\Big|\\&+\left|\frac{b}{a}\right|^2 \Big(\Hessian_\rho (N, N)+\frac{1-\eta}{-\rho}N(\rho)\overline{N} (\rho)\Big)>0.
	\end{split}\] Since $a,b$ are arbitrary, $\left|\frac{b}{a}\right|$ can be any positive number. Let $\xi=|\frac{b}{a}|\geq 0$ and we rewrite $I\!I$ as a quadratic inequality with variable $\xi$: \begin{equation}\label{quad}
		\begin{split}
		&\Big(\Hessian_\rho (L, L)+\frac{1-\eta}{-\rho}|L\rho|^2\Big)-2\xi\Big|\Big(\Hessian_\rho (L, N)+\frac{1-\eta}{-\rho}L\rho\cdot\overline{N}\rho\Big)\Big|\\&+\xi^2 \Big(\Hessian_\rho (N, N)+\frac{1-\eta}{-\rho}N(\rho)\overline{N} (\rho)\Big)>0.
		\end{split}
	\end{equation} We observe that the axis of symmetry is \[\frac{\Big|\Big(\Hessian_\rho (L, N)+\frac{1-\eta}{-\rho}L\rho\cdot\overline{N}\rho\Big)\Big|}{\Big(\Hessian_\rho (N, N)+\frac{1-\eta}{-\rho}N(\rho)\overline{N} (\rho)\Big)}\geq 0.\] Hence, if the quadratic function \[\begin{split}
	&\Big(\Hessian_\rho (L, L)+\frac{1-\eta}{-\rho}|L\rho|^2\Big)-2\xi\Big|\Big(\Hessian_\rho (L, N)+\frac{1-\eta}{-\rho}L\rho\cdot\overline{N}\rho\Big)\Big|\\&+\xi^2 \Big(\Hessian_\rho (N, N)+\frac{1-\eta}{-\rho}N(\rho)\overline{N} (\rho)\Big)
	\end{split}\] has roots, there is one nonnegative root which contradicts (\ref{quad}). Thus, the preceding quadratic function should not have any roots. That means,
	\[\begin{split}
	\Delta=&\Big|\Hessian_\rho (L, N)+\frac{1-\eta}{-\rho}L\rho\cdot\overline{N}\rho\Big|^2\\&-\Big(\Hessian_\rho (L, L)+\frac{1-\eta}{-\rho}|L\rho|^2\Big)\Big(\Hessian_\rho (N, N)+\frac{1-\eta}{-\rho}N(\rho)\overline{N} (\rho)\Big)<0.
	\end{split}\]
	Simplifying $\Delta$, we get \[\begin{split}
	0>\Delta=&|\Hessian_\rho(L, N)|^2+2\frac{1-\eta}{-\rho}\Re\left(\Hessian_\rho(L, N)\overline{L}\rho N\rho\right)-\Hessian_\rho(L, L)\Hessian_\rho(N, N)\\&-\frac{1-\eta}{-\rho}|L\rho|^2\Hessian_\rho(N, N)-(1-\eta)\frac{\Hessian_\rho(L, L)}{-\rho}|N\rho|^2.
	\end{split}\]
	Observe that $\lim\limits_{z\to p}\Hessian_\rho(L, L)=0$ and $\lim\limits_{z\to p}\frac{|L\rho|^2}{\rho}=0$. Letting $z\rightarrow p$, we get the inequality
	\[|\Hessian_\rho(L, N)|^2+2\frac{1-\eta}{-\rho}\Re\left(\Hessian_\rho(L, N)\overline{L}\rho N\rho\right)-(1-\eta)\frac{\Hessian_\rho(L, L)}{-\rho}|N\rho|^2\leq 0.\] The proof is completed by Lemma \ref{calc}.
\end{proof}

For the following lemmas, we have to add a technical assumption. We consider a coordinate chart $U_\alpha$ of $U$. We assume that $L$ is a non-vanishing smooth $(1,0)$-tangent vector field in $U_\alpha$ so that $Lr=0$, where $r$ is a defining function of $\Omega$. This assumption needs localization because there might not be a non-vanishing smooth $(1,0)$-tangent vector field can be found in $U$. The reader is referred to the hairy ball theorem.
\begin{lemma}
		Let $\Omega$ be a bounded pseudoconvex domain with smooth boundary in $\mathbb{C}^n$. Let $U_\alpha$, $L$ and $r$ as above. We denote $N_r$ by $N$. Let \[\Sigma^\alpha_L:=\lbrace p\in U_\alpha\cap\partial\Omega: \Hessian_\rho(L, L)=0 \text{  at  } p\rbrace.\]  If there exists $\epsilon>0$ so that \[\begin{split}
		&\left(\frac{1}{1-\eta}-1\right)\left|(\overline{L}\psi)(Nr)+\Hessian_r(N, L)\right|^2+\frac{\|\nabla r\|}{2}\Bigg(\frac{\|\nabla r\|}{2}\Hessian_\psi(L, L)+g(\nabla_L\nabla_N\nabla r, L)\\&+\|\nabla r\|L\left(\frac{1}{\|\nabla r\|}\right)\Hessian_r(N, L)\Bigg)\leq -\epsilon	 \end{split}
		\]  for all $p\in\Sigma^\alpha_L$, then there exists a neighborhood $V^\alpha_L$ of $\Sigma^\alpha_L$ in $U_\alpha$ so that \[\begin{split}
	&\Hessian_{-(-\rho)^\eta }(aL+bN,aL+bN)\\=&|a|^2\Hessian_{-(-\rho)^\eta}(L, L)+|b|^2\Hessian_{-(-\rho)^\eta }(N, N)+2\Re (a\bar{b}\Hessian_{-(-\rho)^\eta }(L, N))\\=&\eta (-\rho)^{\eta-1}\Bigg(|a|^2\Big(\Hessian_\rho (L, L)+\frac{1-\eta}{-\rho}|L\rho|^2\Big)+2\Re \Big(a\bar{b}\Big(\Hessian_\rho (L, N)+\frac{1-\eta}{-\rho}L\rho\cdot\overline{N}\rho\Big)\Big)\\&+|b|^2 \Big(\Hessian_\rho (N, N)+\frac{1-\eta}{-\rho}N\rho\overline{N} \rho\Big)\Bigg)
	\end{split}\] is positive for all $(a,b)\in\mathbb{C}^2\backslash(0,0)$ and all $z\in\Omega\cap V^\alpha_L$.
\end{lemma}

\begin{proof}
	First, we find that $\overline{L}\rho=0$ on $U_\alpha\cap\partial\Omega$. This implies that $(N-\overline{N})\overline{L}\rho=0$, because $N-\overline{N}$ is a real tangent vector fields on $\partial\Omega$. This implies $N(\overline{L}\rho)=\overline{N}(\overline{L}\rho)$ on $U_\alpha\cap\partial\Omega$. For the same reason, we find out $N|L\rho|^2=\overline{N}|L\rho|^2$ and $N\rho=\overline{N}\rho$ on $U_\alpha\cap\partial\Omega$.
	
	We observe that for $z\in U_\alpha\cap\Omega$ and $p\in U_\alpha\cap\partial\Omega$, \[\lim\limits_{z\to p}\frac{\overline{L}\rho}{\rho}=\lim\limits_{z\to p}\frac{\overline{L}\rho\vert_z-\overline{L}\rho\vert_p}{\rho\vert_z-\rho\vert_p}=\lim\limits_{z\to p}\frac{(N+\overline{N})\overline{L}\rho}{(N+\overline{N})\rho}=\lim\limits_{z\to p}\frac{N\overline{L}\rho}{N\rho}=\frac{N\overline{L}\rho}{N\rho}\Bigg|_p.\]  
	
	 For the same reason, we have that,
	 for $z\in U_\alpha\cap\Omega$ and $p\in U_\alpha\cap\partial\Omega$, \[\lim\limits_{z\to p}\frac{\Hessian_\rho(L, L)\vert_z-\Hessian_\rho(L, L)\vert_p}{\rho\vert_z}=\lim\limits_{z\to p}\frac{N\Hessian_\rho(L, L)}{N\rho}=\frac{N\Hessian_\rho(L, L)}{N\rho}\Bigg|_p,\]
	 and \[\lim\limits_{z\to p}\frac{|L\rho|^2}{\rho}=\lim\limits_{z\to p}\frac{N|L\rho|^2}{N\rho}=\frac{N|L\rho|^2}{N\rho}\Bigg|_p=0.\]  
	 
	 We define the following function $\mathfrak{F}(z)$ in $U_\alpha$ as follows: If $z\in U_\alpha$ and $p\in U_\alpha\cap\partial\Omega$ is closest point to $z$, we have \[\begin{split}
	 \mathfrak{F}(z)=&|\Hessian_\rho(L, N)|^2\vert_z+2\frac{1-\eta}{-\rho\vert_z}\Re\left(\Hessian_\rho(L, N)\overline{L}\rho N\rho\vert_z\right)-(\Hessian_\rho(L, L)\vert_z-\Hessian_\rho(L, L)\vert_p)\Hessian_\rho(N, N)\vert_z\\&-\frac{1-\eta}{-\rho\vert_z}|L\rho|^2\vert_z\Hessian_\rho(N, N)\vert_z-(1-\eta)\frac{\Hessian_\rho(L, L)\vert_z-\Hessian_\rho(L, L)\vert_p}{-\rho\vert_z}|N\rho|^2\vert_z,
	 \end{split}\]
	 and if $z\in U_\alpha\cap\partial\Omega$, we have \[\begin{split}
	 \mathfrak{F}(z)=&|\Hessian_\rho(L, N)|^2-2(1-\eta)\Re\left(\Hessian_\rho(L, N)\frac{\overline{L}\rho}{\rho} N\rho\right)-(1-\eta)\frac{N\Hessian_\rho(L, L)}{-N\rho}|N\rho|^2.
	 \end{split}\]
	 By the discussion above, we see that $\mathfrak{F}(z)$ is continuous in $U_\alpha$. 
	 
	 By the assumption that \[\begin{split}
	 &\left(\frac{1}{1-\eta}-1\right)\left|(\overline{L}\psi)(Nr)+\Hessian_r(N, L)\right|^2+\frac{\|\nabla r\|}{2}\Bigg(\frac{\|\nabla r\|}{2}\Hessian_\psi(L, L)+g(\nabla_L\nabla_N\nabla r, L)\\&+\|\nabla r\|L\left(\frac{1}{\|\nabla r\|}\right)\Hessian_r(N, L)\Bigg)\leq -\epsilon	 \end{split}
	 \] on $\Sigma^\alpha_L$, we can find a neighborhood $V^\alpha_L$ of $\Sigma^\alpha_L$ in $U_\alpha$ so that $\mathfrak{F}(z)<-\frac{\epsilon}{2}$ in $V^\alpha_L$. In particular, it means  \[\begin{split}
	 &|\Hessian_\rho(L, N)|^2\vert_z+2\frac{1-\eta}{-\rho\vert_z}\Re\left(\Hessian_\rho(L, N)\overline{L}\rho N\rho\vert_z\right)-(\Hessian_\rho(L, L)\vert_z-\Hessian_\rho(L, L)\vert_p)\Hessian_\rho(N, N)\vert_z\\&-\frac{1-\eta}{-\rho\vert_z}|L\rho|^2\vert_z\Hessian_\rho(N, N)\vert_z-(1-\eta)\frac{\Hessian_\rho(L, L)\vert_z-\Hessian_\rho(L, L)\vert_p}{-\rho\vert_z}|N\rho|^2\vert_z<0,
	 \end{split}\] for $z\in V^\alpha_L\cap\Omega$. But this exactly means $\Delta<0$ for the following quadratic function:
	 \[\begin{split}
	 &\Big((\Hessian_\rho(L, L)\vert_z-\Hessian_\rho(L, L)\vert_p)+\frac{1-\eta}{-\rho\vert_z}|L\rho|^2\vert_z\Big)-2\xi \Big|\Hessian_\rho (L, N)\vert_z+\frac{1-\eta}{-\rho\vert_z}L\rho\vert_z\cdot\overline{N}\rho\vert_z\Big|\\&+\xi^2 \Big(\Hessian_\rho (N, N)\vert_z+\frac{1-\eta}{-\rho\vert_z}N(\rho)\vert_z\overline{N} (\rho)\vert_z\Big).
	 \end{split}\] 
	 
	 Hence,  \[\begin{split}
	 &\Big((\Hessian_\rho(L, L)\vert_z-\Hessian_\rho(L, L)\vert_p)+\frac{1-\eta}{-\rho\vert_z}|L\rho|^2\vert_z\Big)-2\xi \Big|\Hessian_\rho (L, N)\vert_z+\frac{1-\eta}{-\rho\vert_z}L\rho\vert_z\cdot\overline{N}\rho\vert_z\Big|\\&+\xi^2 \Big(\Hessian_\rho (N, N)\vert_z+\frac{1-\eta}{-\rho\vert_z}N(\rho)\vert_z\overline{N} (\rho)\vert_z\Big)>0,
	 \end{split}\] and by replacing $\xi$ with $\left|\frac{b}{a}\right|^2$, we obtain that,
	 \[\begin{split}
	 &|a|^2\Big((\Hessian_\rho(L, L)\vert_z-\Hessian_\rho(L, L)\vert_p)+\frac{1-\eta}{-\rho\vert_z}|L\rho|^2\vert_z\Big)-2|ab| \Big|\Hessian_\rho (L, N)\vert_z+\frac{1-\eta}{-\rho\vert_z}L\rho\vert_z\cdot\overline{N}\rho\vert_z\Big|\\&+|b|^2 \Big(\Hessian_\rho (N, N)\vert_z+\frac{1-\eta}{-\rho\vert_z}N(\rho)\vert_z\overline{N} (\rho)\vert_z\Big)>0
	 \end{split}\] for any $(a,b)\in\mathbb{C}^2\backslash(0,0)$. This implies that \[\begin{split}
	 &|a|^2\Big(\Hessian_\rho(L, L)\vert_z+\frac{1-\eta}{-\rho\vert_z}|L\rho|^2\vert_z\Big)-2|ab| \Big|\Hessian_\rho (L, N)\vert_z+\frac{1-\eta}{-\rho\vert_z}L\rho\vert_z\cdot\overline{N}\rho\vert_z\Big|\\&+|b|^2 \Big(\Hessian_\rho (N, N)\vert_z+\frac{1-\eta}{-\rho\vert_z}N(\rho)\vert_z\overline{N} (\rho)\vert_z\Big)>0
	 \end{split}\] for all $z\in V^\alpha_L\cap\Omega$ and $(a,b)\in\mathbb{C}^2\backslash(0,0)$ because $\Hessian_\rho(L, L)\vert_p\geq 0$ on pseudoconvex domains. This completes the proof.
\end{proof}

Let $\Sigma$ be the set of points with degenerate Levi-forms. We derive some lemmas with the extra assumption $g(L, L)=\frac{1}{2}$. This normalization is not essential, however, it makes the following proofs simpler.

\begin{lemma}
	Let $\Omega$ be a bounded pseudoconvex domain with smooth boundary in $\mathbb{C}^n$. We consider a coordinate chart of $U_\alpha$ of $U$. We assume that $L$ is a non-vanishing smooth $(1,0)$-tangent vector field in $U_\alpha$ so that $Lr=0$. We also assume the normalization $g(L, L)=\frac{1}{2}$. We denote $N_r$ by $N$. Let \[\Sigma^\alpha_L:=\lbrace p\in U_\alpha\cap\partial\Omega: \Hessian_\rho(L, L)=0 \text{  at  } p\rbrace.\] Suppose that $\psi$ is defined in a neighborhood of $\Sigma$ in $\partial\Omega$ and on $\Sigma^\alpha_L$
	\[\begin{split}
	&\left(\frac{1}{1-\eta}-1\right)\left|(\overline{L}\psi)(Nr)+\Hessian_r(N, L)\right|^2+\frac{\|\nabla r\|}{2}\Bigg(\frac{\|\nabla r\|}{2}\Hessian_\psi(L, L)+g(\nabla_L\nabla_N\nabla r, L)\\&+\|\nabla r\|L\left(\frac{1}{\|\nabla r\|}\right)\Hessian_r(N, L)\Bigg)\leq 0	 \end{split}
	\]  Then for any $\nu>0$, there exists $\mu<0$ so that on $\Sigma^\alpha_L$,
	\[\begin{split}
	&\left(\frac{1}{1-\tilde{\eta}}-1\right)\left|(\overline{L}\tilde{\psi})(Nr)+\Hessian_r(N, L)\right|^2+\frac{\|\nabla r\|}{2}\Bigg(\frac{\|\nabla r\|}{2}\Hessian_{\tilde{\psi}}(L, L)+g(\nabla_L\nabla_N\nabla r, L)\\&+\|\nabla r\|L\left(\frac{1}{\|\nabla r\|}\right)\Hessian_r(N, L)\Bigg)\leq \frac{\mu}{4},
	\end{split}\] where $\tilde{\psi}=\psi+\mu|z|^2$ for some $\mu<0$ and $\tilde{\eta}=1-\frac{1}{\frac{1}{1-\eta}-\nu}$.
\end{lemma}
\begin{proof}
	Let $\tilde{\psi}=\psi+\mu|z|^2$, where $\mu$ will be decided later. We calculate 
	\[\Hessian_{\tilde{\psi}}(L, L)=\Hessian_\psi(L, L)+\Hessian_{\mu|z|^2}(L, L)=\Hessian_\psi(L, L)+\frac{\mu}{2}.\]
	
	We also calculate \[\begin{split}
	&\left(\frac{1}{1-\tilde{\eta}}-1\right)\left|(\overline{L}\tilde{\psi})(Nr)+\Hessian_r(N, L)\right|^2\\=&\left(\frac{1}{1-\tilde{\eta}}-1\right)\left|(\overline{L}\psi)(Nr)+\Hessian_r(N, L)\right|^2+\left(\frac{1}{1-\tilde{\eta}}-1\right)\left|\mu(Nr)\overline{L}|z|^2\right|^2\\&+\left(\frac{1}{1-\tilde{\eta}}-1\right)\mu\Re\left(L|z|^2\cdot\left((\overline{L}\psi)(Nr)+\Hessian_r(N, L)\right)\right)\\=&\left(\frac{1}{1-\eta}-1\right)\left|(\overline{L}\psi)(Nr)+\Hessian_r(N, L)\right|^2-\nu\left|(\overline{L}\psi)(Nr)+\Hessian_r(N, L)\right|^2\\&+\left(\frac{1}{1-\eta}-1-\nu\right)\left|\mu(Nr)\overline{L}|z|^2\right|^2+\left(\frac{1}{1-\eta}-1-\nu\right)\mu\Re\left(L|z|^2\cdot\left((\overline{L}\psi)(Nr)+\Hessian_r(N, L)\right)\right)\\=&\left(\frac{1}{1-\eta}-1\right)\left|(\overline{L}\psi)(Nr)+\Hessian_r(N, L)\right|^2-\nu\left|(\overline{L}\psi)(Nr)+\Hessian_r(N, L)\right|^2\\&+\left(\frac{1}{1-\eta}-1-\nu\right)\left|\mu(Nr)\overline{L}|z|^2\right|^2+\frac{\nu}{2}\left|(\overline{L}\psi)(Nr)+\Hessian_r(N, L)\right|^2+\frac{\mu^2\left(\frac{1}{1-\eta}-1-\nu\right)^2}{2\nu}\left|L|z|^2\right|^2\\\leq&\left(\frac{1}{1-\eta}-1\right)\left|(\overline{L}\psi)(Nr)+\Hessian_r(N, L)\right|^2+\mu^2\left|\overline{L}|z|^2\right|^2\left(\frac{1}{1-\eta}-1-\nu\right)\left(|Nr|^2+\frac{\frac{1}{1-\eta}-1-\nu}{2\nu}\right).
	\end{split}\]
	
	Let $T_j$ be smooth $(1,0)$ tangent vector fields in $U_\alpha$ so that $T_j r=0$, for $0\leq j\leq n-2$. We assume also $\lbrace \sqrt{2}T_j\rbrace_{j=0}^{n-2}$ be orthonormal and we define \[M:=\left(\frac{1}{1-\eta}-1-\nu\right)\max_{\alpha}\max_{z\in U_\alpha\cap\Omega}\left(\sum_{j=0}^{n-2}\left|T_j|z|^2\right|\right)^2\left(|Nr|^2+\frac{\frac{1}{1-\eta}-1-\nu}{2\nu}\right),\] which is independent from $U_\alpha$. By compactness, we have only finite many $U_\alpha$. Thus $0\leq M<\infty$. Then we find that \[\begin{split}
	&\left(\frac{1}{1-\tilde{\eta}}-1\right)\left|(\overline{L}\tilde{\psi})(Nr)+\Hessian_r(N, L)\right|^2+\frac{\|\nabla r\|}{2}\left(\frac{\|\nabla r\|}{2}\Hessian_{\tilde{\psi}}(L, L)+g(\nabla_L\nabla_N\nabla r, L)\right)\\\leq&\left(\frac{1}{1-\eta}-1\right)\left|(\overline{L}\psi)(Nr)+\Hessian_r(N, L)\right|^2+\frac{\|\nabla r\|}{2}\left(\frac{\|\nabla r\|}{2}\Hessian_\psi(L, L)+g(\nabla_L\nabla_N\nabla r, L)\right)+\frac{\mu}{2}+M\mu^2\\\leq&\frac{\mu}{2}+M\mu^2-\frac{\|\nabla r\|^2}{2}L\left(\frac{1}{\|\nabla r\|}\right)\Hessian_r(N, L)\leq\frac{\mu}{4}-\frac{\|\nabla r\|^2}{2}L\left(\frac{1}{\|\nabla r\|}\right)\Hessian_r(N, L),
	\end{split}\] if we take $\mu\in (-\frac{1}{4M},0)$.
\end{proof}

\begin{lemma}\label{2.5}
		Let $\Omega$ be a bounded pseudoconvex domain with smooth boundary in $\mathbb{C}^n$. We consider a coordinate chart of $U_\alpha$ of $U$. We assume that $L$ is a non-vanishing smooth $(1,0)$-tangent vector field in $U_\alpha$ so that $Lr=0$. We also assume the normalization $g(L, L)=\frac{1}{2}$. We denote $N_r$ by $N$. Let \[\Sigma^\alpha_L:=\lbrace p\in U_\alpha\cap\partial\Omega: \Hessian_\rho(L, L)=0 \text{  at  } p\rbrace.\] Suppose that $\psi$ is defined in a neighborhood of $\Sigma$ in $\partial\Omega$ and on $\Sigma^\alpha_L$
	\[\begin{split}
	&\left(\frac{1}{1-\eta}-1\right)\left|(\overline{L}\psi)(Nr)+\Hessian_r(N, L)\right|^2+\frac{\|\nabla r\|}{2}\Bigg(\frac{\|\nabla r\|}{2}\Hessian_\psi(L, L)+g(\nabla_L\nabla_N\nabla r, L)\\&+\|\nabla r\|L\left(\frac{1}{\|\nabla r\|}\right)\Hessian_r(N, L)\Bigg)\leq 0.
	\end{split}
	\] For any $\nu>0$, we let $\tilde{\eta}=1-\frac{1}{\frac{1}{1-\eta}-\nu}$. Then there exists $\tilde{\rho}$ so that  \[\Hessian_{-(-\tilde{\rho})^{\tilde{\eta}}}(aL+bN,aL+bN)>0\] in all $U_\alpha\cap\Omega$ and for all $(a,b)\in\mathbb{C}^2\backslash(0,0)$.
\end{lemma}

\begin{proof}
	Using a partition of unity, we can extend $\tilde{\psi}$ from the previous lemma to $\cup_{\alpha}U_\alpha$, because $\tilde{\psi}$ is defined independently of $\alpha$. We still denote the extension by $\tilde{\psi}$. Let $\tilde{\rho}=r e^{\tilde{\psi}}$. By the preceding two lemmas, we just need to show that  \[\Hessian_{-(-\tilde{\rho})^{\tilde{\eta}}}(aL+bN,aL+bN)>0\] in $U_\alpha\backslash V_L^\alpha$ for all $(a,b)\in\mathbb{C}^2\backslash(0,0)$. In $(U_\alpha\cap\partial\Omega)\backslash V_L^\alpha$, we can see there exists $c>0$ so that \[\Hessian_{\tilde{\rho}}(L, L)>c.\] Inspect the quadratic polynomial as before,
	\[\begin{split}
	&\Big(\Hessian_{\tilde{\rho}}(L, L)+\frac{1-\eta}{-\tilde{\rho}}|L\tilde{\rho}|^2\Big)-2|\xi| \Big|\Hessian_{\tilde{\rho}} (L, N)+\frac{1-\eta}{-\tilde{\rho}}L\tilde{\rho}\cdot\overline{N}\tilde{\rho}\Big|\\&+|\xi|^2 \Big(\Hessian_{\tilde{\rho}} (N, N)+\frac{1-\eta}{-\tilde{\rho}}N(\tilde{\rho})\overline{N} (\tilde{\rho})\Big).
	\end{split}\]
	Note that $\Delta$ is given by \[\begin{split}
	&\Big|\Hessian_{\tilde{\rho}} (L, N)+\frac{1-\eta}{-\tilde{\rho}}L\tilde{\rho}\cdot\overline{N}\tilde{\rho}\Big|^2-\Big(\Hessian_{\tilde{\rho}}(L, L)+\frac{1-\eta}{-\tilde{\rho}}|L\tilde{\rho}|^2\Big)\Big(\Hessian_{\tilde{\rho}} (N, N)+\frac{1-\eta}{-\tilde{\rho}}N(\tilde{\rho})\overline{N} (\tilde{\rho})\Big)\\\leq&\Big|\Hessian_{\tilde{\rho}} (L, N)+\frac{1-\eta}{-\tilde{\rho}}L\tilde{\rho}\cdot\overline{N}\tilde{\rho}\Big|^2-\frac{c}{2}\Big(\Hessian_{\tilde{\rho}} (N, N)+\frac{1-\eta}{-\tilde{\rho}}N(\tilde{\rho})\overline{N} (\tilde{\rho})\Big),
	\end{split}\]
	in an (inside) tubular neighborhood of $U_\alpha\cap\partial\Omega$ in $U_\alpha$. This is because $\frac{|L\tilde{\rho}|^2}{-\tilde{\rho}}$ approaches $0$ as $z$ approaches $\partial\Omega$. By possibly shrinking the (inside) tubular neighborhood, this gives $\Delta<0$ because $\Big|\Hessian_{\tilde{\rho}} (L, N)+\frac{1-\eta}{-\tilde{\rho}}L\tilde{\rho}\cdot\overline{N}\tilde{\rho}\Big|$ and $\Hessian_{\tilde{\rho}}(N, N)$ are bounded while $\frac{1-\eta}{-\tilde{\rho}}N(\tilde{\rho})\overline{N} (\tilde{\rho})$ blows up to $\infty$, as $z$ approaches $\partial\Omega$. This completes the proof.
\end{proof}

\begin{definition}[The Diederich--Forn\ae ss index for hypersurface.]\label{dfb}
	Let $\Pi$ be a bounded pseudoconvex smooth hypersurface in $\mathbb{C}^n$,  $L$ be an arbitrary smooth $(1,0)$ tangent vector field on $\Pi$ and $r$ be the defining function of $\Pi$. Let \[\Sigma_L:=\lbrace p\in\Pi: \Hessian_r(L, L)=0 \text{  at  } p\rbrace.\] The Diederich--Forn\ae ss index for $\Pi$ is $\sup_\psi \eta_\psi$. Here $\eta_\psi\in(0,1)$ is a number affiliated to the smooth function $\psi$ which is defined on a neighborhood of $\Sigma$ in $\Pi$, so that on $\Sigma_L$,
	\[\begin{split}
		&\left(\frac{1}{1-\eta_\psi}-1\right)\left|(\overline{L}\psi)(N_rr)+\Hessian_r(N_r, L)\right|^2+\frac{\|\nabla r\|}{2}\Bigg(\frac{\|\nabla r\|}{2}\Hessian_\psi(L, L)+g(\nabla_L\nabla_{N_r}\nabla r, L)\\&+\|\nabla r\|L\left(\frac{1}{\|\nabla r\|}\right)\Hessian_r(N_r, L)\Bigg)\leq 0,
	\end{split}
	\] for all $L$.
\end{definition}

\begin{remark}
	When the defining function $r$ has constant gradient norm $\|\nabla r\|$ on the set of degenerated Levi-forms, then the condition of the Diederich--Forn\ae ss index of hypersurface can be simplified to 	\[
	\left(\frac{1}{1-\eta_\psi}-1\right)\left|(\overline{L}\psi)(N_rr)+\Hessian_r(N_r, L)\right|^2+\frac{\|\nabla r\|}{2}\Bigg(\frac{\|\nabla r\|}{2}\Hessian_\psi(L, L)+g(\nabla_L\nabla_{N_r}\nabla r, L)\Bigg)\leq 0.
	\] 
\end{remark}

\begin{theorem}\label{equiv}
Let $\Omega$ be a bounded pseudoconvex domain with smooth boundary in $\mathbb{C}^n$. The Diederich--Forn\ae ss index for $\Omega$ is the Diederich--Forn\ae ss index for $\partial\Omega$.
\end{theorem}

\begin{proof}
	We denote $N_r$ by $N$. For this proof, let $\eta_0$ be the Diederich--Forn\ae ss index of $\Omega$ and $\eta_1$ be the Diederich--Forn\ae ss index for $\partial\Omega$. We know for any $\eta<\eta_0$ there exists $\rho$ so that $-(-\rho)^\eta$ is plurisubharmonic. Particularly, for any smooth $(1,0)$ tangent vector field $L$ on $U$, we have \[\Hessian_{-(-\rho)^\eta}(aL+bN, aL+bN)>0\] for all $(a,b)\in\mathbb{C}^2\backslash(0,0)$. By Lemma \ref{2.2}, it implies\[\begin{split}
	&\left(\frac{1}{1-\eta}-1\right)\left|(\overline{L}\psi)(Nr)+\Hessian_r(N, L)\right|^2+\frac{\|\nabla r\|}{2}\Bigg(\frac{\|\nabla r\|}{2}\Hessian_\psi(L, L)+g(\nabla_L\nabla_N\nabla r, L)\\&+\|\nabla r\|L\left(\frac{1}{\|\nabla r\|}\right)\Hessian_r(N, L)\Bigg)\leq 0\end{split}
	\] on $\Sigma_L$. We proved that $\eta_1$ is greater than $\eta_0$. 
	
	We are going to show that $\eta_0$ is greater than $\eta_1$. By Lemma \ref{2.5}, we find out for any $\eta<\eta_1$, we have \[\Hessian_{-(-\rho)^{\tilde{\eta}}}(aL+bN,aL+bN)>0\] on all $U_\alpha\cap\Omega$, for all smooth $(1, 0)$ tangent vector fields in $U$ satisfying $g(L, L)=\frac{1}{2}$ in $U_\alpha$ and all $(a,b)\in\mathbb{C}^2\backslash(0,0)$, where $\tilde{\eta}=\frac{1}{1-\eta}-\nu$ for arbitrary small $\nu>0$. Thus $-(-\rho)^{\tilde{\eta}}$ is plurisubharmonic near $\partial\Omega$ in $\Omega$. By compactness and a filling a hole argument (see \citep{DF77b}), since $\nu$ is arbitrarily small, we know that $\eta_0$ is greater than $\eta_1$, which completes the proof.
\end{proof}

In $\mathbb{C}^2$, the hairy ball theorem cannot happen. Indeed,  \[L=\frac{1}{\sqrt{|\frac{\partial r}{\partial z}|^2+|\frac{\partial r}{\partial w}|^2}}(\frac{\partial r}{\partial w}\frac{\partial}{\partial z}-\frac{\partial r}{\partial z}\frac{\partial}{\partial w})\] is a well globally-defined $(1,0)$ tangent vector field in $U$. It also satisfies $g(L, L)=\frac{1}{2}$. Hence, for $\mathbb{C}^2$, we have the following corollary.

\begin{corollary}\label{c2}
	Let $\Omega$ be a bounded pseudoconvex domain with smooth boundary in $\mathbb{C}^2$ defined by $r$. Let $\Sigma$ be the set of Levi-flat points. The Diederich--Forn\ae ss index of $\Omega$ is $\eta_0$ if and only if for any $\eta\in(0, \eta_0)$, there exists a smooth $\psi$ defined on a neighborhood of $\Sigma$ in $\partial\Omega$ so that on $\Sigma$, \[\begin{split}
	&\left(\frac{1}{1-\eta}-1\right)\left|(\overline{L}\psi)(N_rr)+\Hessian_r(N_r, L)\right|^2+\frac{\|\nabla r\|}{2}\Bigg(\frac{\|\nabla r\|}{2}\Hessian_\psi(L, L)+g(\nabla_L\nabla_{N_r}\nabla r, L)\\&+\|\nabla r\|L\left(\frac{1}{\|\nabla r\|}\right)\Hessian_r(N_r, L)\Bigg)\leq 0,
	\end{split}
	\] holds for one smooth non-vanishing $(1,0)$ tangent vector field $L$ of $\partial\Omega$.
\end{corollary}

\section{Diederich--Forn\ae ss index of $\beta$-worm domains}\label{sec3}

\begin{definition}\label{worm}
	The $\beta$- worm domain $\Omega_\beta$ is defined by \[\lbrace (z, w)\in\mathbb{C}^2: r(z, w)=\left|z-e^{i\log|w|^2}\right|^2<1-\eta(\log|w|^2)\rbrace,\] where $\eta: \mathbb{R}\mapsto\mathbb{R}$ is a fixed smooth function with the following properties:	\begin{enumerate}
		\item $\eta(x)\geq 0$, $\eta$ is even and convex.
		\item $\eta^{-1}(0)=I_{\beta-\pi/2}=[-\beta+\pi/2, \beta-\pi/2]$.
		\item there exists and $a>0$ such that $\eta(x)>1$ if $x<-a$ or $x>a$.
		\item $\eta'(x)\neq 0$ if $\eta(x)=1$.
	\end{enumerate}
\end{definition}

In this section, we seek the Diederich--Forn\ae ss index of $\beta$-worm domain. We can see that by varying $\beta$, the worm domains indeed exhaust all possibilities of non-trivial indexes. 

Let $\Sigma$ be Levi-flat set of $\partial\Omega_\beta$. It is well-known that, \[\Sigma=\lbrace (0, w): \left|\log|w|^2\right|\leq \beta-\frac{\pi}{2}\rbrace.\] For the following standard calculations of $\beta$-worm domains, readers are referred to \citep{DF77a} and \citep{KP08}. By a straightforward calculations, one can see that the following: \begin{enumerate}
	\item \[\frac{\partial r}{\partial z}=\bar{z}+e^{-i\log|w|^2}\qquad\text{and}\qquad\frac{\partial r}{\partial w}=\frac{2i}{w}\Re(\bar{z}e^{i\log|w|^2});\] 
	\item \[\frac{\partial^2 r}{\partial z\partial\bar{z}}=1,\qquad\frac{\partial^2 r}{\partial z\partial\bar{w}}=\frac{-i}{\bar{w}}e^{-i\log|w|^2}\quad\text{and}\quad\frac{\partial^2 r}{\partial w\partial\bar{w}}=\frac{-2}{|w|^2}\Re(\bar{z}e^{i\log|w|^2}).\]
\end{enumerate}  On $\Sigma$, $z=0$. We can further simplify the calculation above and obtain that
\begin{enumerate}
	\item \[\frac{\partial r}{\partial z}=e^{-i\log|w|^2}\qquad\text{and}\qquad\frac{\partial r}{\partial w}=0;\] 
	\item \[\frac{\partial^2 r}{\partial z\partial\bar{z}}=1,\qquad\frac{\partial^2 r}{\partial z\partial\bar{w}}=\frac{-i}{\bar{w}}e^{-i\log|w|^2}\quad\text{and}\quad\frac{\partial^2 r}{\partial w\partial\bar{w}}=0.\]
\end{enumerate}

This gives, 
\[N_r=e^{i\log|w|^2}\frac{\partial}{\partial z}\qquad\text{and}\qquad\|\nabla r\|=2.\]  

 Since $\beta$-worm domains are defined in $\mathbb{C}^2$, by Corollary \ref{c2}, we have the following lemma.
 
 \begin{lemma}
 	Let $\Omega_\beta$, $r$ and $\Sigma$ defined above. The Diederich--Forn\ae ss index of $\Omega_\beta$ is $\eta_0$ if and only if for any $\eta\in (0,\eta_0)$, there exists a smooth function defined in a neighborhood of $\Sigma$ in $\partial\Omega$ so that on $\Sigma$, \begin{equation}\label{start}
 	\left(\frac{1}{1-\eta}-1\right)\left|(\overline{L}\psi)(N_rr)+\Hessian_r(N_r, L)\right|^2+\frac{\|\nabla r\|}{2}\left(\frac{\|\nabla r\|}{2}\Hessian_\psi(L, L)+g(\nabla_L\nabla_{N_r}\nabla r, L)\right)\leq 0,
 	\end{equation}
 	holds for \[L=e^{i\log |w|^2}\left(\frac{\partial r}{\partial w}\frac{\partial}{\partial z}-\frac{\partial r}{\partial z}\frac{\partial}{\partial w}\right).\]
 \end{lemma}

We calculate the quantities in the preceding lemma. On $\Sigma$, 
 \[L=\frac{\partial}{\partial w},\qquad\Hessian_r(N_r, L)=\frac{-i}{\bar{w}}\quad\text{and}\quad g(\nabla_L\nabla_{N_r}\nabla r, L)=0.\] Thus defining $\alpha:=\frac{1}{1-\eta}-1$, (\ref{start}) can be simplified to 
 \begin{equation}\label{cont}
 	\alpha\left|\frac{\partial\psi}{\partial \bar{w}}-\frac{i}{\bar{w}}\right|^2+\frac{1}{4}\Delta\psi\leq 0.
 \end{equation}

To find the optimal $\alpha$, we need a basic lemma from ordinary differential equations.

\begin{lemma}[Riccati equations]\label{riccati}
	Consider the following type Riccati equation: \[s'(t)=-as^2(t)-\frac{s(t)}{t}-\frac{b}{t^2}\] for $a,b>0$ and $t>0$. Then the solution is \[s=\sqrt{\frac{b}{a}}\cdot\frac{\cot(\sqrt{ab}\log t+\theta)}{t}\] for arbitrary $\theta$.
\end{lemma}

\begin{proof}
	First, we make the substitution: $as=\frac{u'}{u}$. It becomes \[u''+\frac{u'}{t}+\frac{abu}{t^2}=0.\]This is equivalent to the second-order Euler equation:\[t^2+tu'+abu=0.\] Thus the solution is \[u=C_1\sin(\sqrt{ab}\log t)+C_2\cos(\sqrt{ab}\log t),\] where $C_1, C_2$ are arbitrary constants. 
	
	We simplify the solution above \[u=C_1\sin(\sqrt{ab}\log t)+C_2\cos(\sqrt{ab}\log t)=\sqrt{C_1^2+C_2^2}\cdot\sin(\sqrt{ab}\log t+\theta),\] where $\cos\theta=\frac{C_1}{\sqrt{C_1^2+C_2^2}}$ and $\sin\theta=\frac{C_2}{\sqrt{C_1^2+C_2^2}}$. Thus \[s=\sqrt{\frac{b}{a}}\cdot\frac{\cot(\sqrt{ab}\log t+\theta)}{t}.\]
\end{proof}

We now use polar coordinates. Let $w=re^{i\phi}$, and then by (\ref{cont}), we obtain that \[\frac{\alpha f_r^2}{4}+\frac{\alpha f_\phi^2}{4r^2}+\frac{\alpha}{r^2}-\frac{\alpha f_\phi}{r^2}+\frac{f_{rr}}{4}+\frac{f_r}{4r}+\frac{f_{\phi\phi}}{4r^2}\leq 0.\] Here we replace the function $\psi$ in (\ref{cont}) with $f$ to avoid the confusion of $\psi$ and the angle variable $\phi$.

Integral both sides and we have that 
\[\int_{0}^{2\pi}\frac{\alpha f_r^2}{4}\, d\phi+\int_{0}^{2\pi}\frac{\alpha f_\phi^2}{4r^2}\, d\phi+\int_{0}^{2\pi}\frac{\alpha}{r^2}\, d\phi-\int_{0}^{2\pi}\frac{\alpha f_\phi}{r^2}\, d\phi+\int_{0}^{2\pi}\frac{f_{rr}}{4}\, d\phi+\int_{0}^{2\pi}\frac{f_r}{4r}\, d\phi+\int_{0}^{2\pi}\frac{f_{\phi\phi}}{4r^2}\, d\phi\leq 0.\]
By Schwarz's lemma
\[\int_{0}^{2\pi}\,d\phi\cdot \int_{0}^{2\pi}f_r^2\,d\phi\geq\left(\int_{0}^{2\pi}f_r\,d\phi\right)^2,\]
and note that $\int_{0}^{2\pi}f_\phi\,d\phi=\int_{0}^{2\pi} f_{\phi\phi}\, d\phi=0$. We then have that 
\[\frac{\alpha}{8\pi}\left(\int_{0}^{2\pi}f_r\, d\phi\right)^2+\frac{2\alpha\pi}{r^2}+\frac{1}{4}\int_{0}^{2\pi}f_{rr}\, d\phi+\frac{1}{4r}\int_{0}^{2\pi}f_r\, d\phi\leq 0.\]

We define that $F(r)=\frac{1}{2\pi}\int_{0}^{2\pi} f(r,\phi)\, d\phi$. It is clear that $F(r)$ is smooth on $\Sigma$. That means $F_r$ is smooth on $[e^{\frac{\pi}{4}-\frac{\beta}{2}}, e^{\frac{\beta}{2}-\frac{\pi}{4}}]$. Interchanging $\frac{\partial}{\partial r}$ and integral sign, we have that
\[\frac{2\alpha\pi}{4} F_r^2+\frac{2\alpha\pi}{r^2}+\frac{2\pi F_{rr}}{4}+\frac{2\pi F_r}{4r}\leq 0.\]
Let $s=F_r$, we have that
\[s'+\alpha s^2+\frac{s}{r}+\frac{4\alpha}{r^2}\leq 0.\]
By the comparison principal of ordinary differential equation, suppose that $s(r_0)=s_0$, we have that 
\[s\leq 2\frac{\cot(2\alpha\log r+\theta)}{r},\] where $\theta$ is a constant such that \[ 2\frac{\cot(2\alpha\log r_0+\theta)}{r_0}=s_0.\] Since the cotangent function has a period $\pi$, by $\frac{\pi}{4}-\frac{\beta}{2}\leq\log r\leq\frac{\beta}{2}-\frac{\pi}{4}$, we have that $2\alpha(\beta-\frac{\pi}{2})<\pi$. Thus we obtain that $\alpha<\frac{\pi}{2\beta-\pi}$. The corresponding $\eta$ satisfies $\eta<\pi/{(2\beta)}$. Thus we improved the upper bound of Diederich--Forn\ae ss index of $\Omega_\beta$. It should be less or equal to $\pi/{(2\beta)}$.

We are going to show this constant is sharp. If for any $\alpha<\frac{\pi}{2\beta-\pi}$, we can construct a bounded smooth solution of equation \[\frac{\alpha f_r^2}{4}+\frac{\alpha f_\phi^2}{4r^2}+\frac{\alpha}{r^2}-\frac{\alpha f_\phi}{r^2}+\frac{f_{rr}}{4}+\frac{f_r}{4r}+\frac{f_{\phi\phi}}{4r^2}=0,\] then we are done. This is because we can extend this function to a neighborhood of $\Sigma$. We will suppose $f(r,\phi)$ only depends on $r$. Then the equation becomes \[\frac{\alpha f_r^2}{4}+\frac{\alpha}{r^2}+\frac{f_{rr}}{4}+\frac{f_r}{4r}=0.\] We define $s=f_r$ and this differential equation becomes a Riccati equation again. The solution for $s$ is \[s=2\frac{\cot(2\alpha\log r+\theta)}{r}\] for some $\theta$. This is well-defined and bounded on $\Sigma$ because $\alpha<\frac{\pi}{2\beta-\pi}$. Hence, we have the following theorem.

\begin{theorem}\label{calcworm}
	Let $\Omega_\beta$ be the $\beta$-worm domain. The Diederich--Forn\ae ss index of $\Omega_\beta$ is $\pi/{(2\beta)}$.
\end{theorem}

\begin{remark}
	This is the first time we have found accurate non-trivial Diederich--Forn\ae ss indexes in Euclidean spaces. This also shows the following statement: for arbitrary $\eta_0\in(0,1)$, $\Omega_{\pi/2\eta_0}$ is a bounded pseudoconvex domain of index $\eta_0$. Thus the Diederich--Forn\ae ss index is a continuum, not discrete.
\end{remark}

	\bigskip
	\bigskip
	\noindent {\bf Acknowledgments}. The author thanks to Dr. Steven Krantz for drawing his attention to this topic. The author thanks to Dr. Mei-Chu Chang, Dr. Siqi Fu and Dr. Michael Hartglass for their advice. The author also thanks to Dr. Masanori Adachi, Dr. Xinghong Pan, Dr. Marco Peloso, Dr. Lihan Wang, Dr. Bun Wong, Dr. Yuan Yuan, Dr. Qi S. Zhang and Dr. Meng Zhu for fruitful conversations. Last but not least, the author thank to Dr. Jeffery McNeal for helping make the references more accurate.
	
	\printbibliography
	
\end{document}